\documentclass[12pt]{amsart}
\usepackage{amssymb,longtable}
\theoremstyle{definition}
\newtheorem{thm}{Theorem}[section]
\newtheorem{prop}[thm]{Proposition}
\newtheorem{df}[thm]{Definition}
\newtheorem{lem}[thm]{Lemma}
\newtheorem{rem}[thm]{Remark}
\newtheorem{cor}[thm]{Corollary}
\newtheorem{ex}[thm]{Example}
\newtheorem*{ack}{Acknowledgments}
\newtheorem*{mt}{Main Theorem}

\def\ord{\mathop{\mathrm{ord}}\nolimits}
\def\rk{\mathop{\mathrm{rank}}\nolimits}

\def\Hom{\mathop{\mathrm{Hom}}\nolimits}
\title[Non-symplectic automorphisms on $K3$ surfaces]
{Classification of non-symplectic automorphisms on $K3$ surfaces which act trivially on the N\'{e}ron-Severi lattice}
\author[S.~Taki]{Shingo Taki}
\address{School of Information Environment, Tokyo Denki University,
2-1200 Muzai Gakuendai, Inzai-shi, Chiba 270-1382, Japan}
\email{taki@sie.dendai.ac.jp}
\date{\today}
\begin{document}

\begin{abstract}
We treat non-symplectic automorphisms on $K3$ surfaces 
which act trivially on the N\'{e}ron-Severi lattice.
In this paper, we classify non-symplectic automorphisms of prime-power order, 
especially 2-power order on $K3$ surfaces, i.e., we describe their fixed locus.
\end{abstract}

\maketitle

\section{Introduction}\label{Introduction}
Let $X$ be a $K3$ surface. 
In the following, we denote by $S_{X}$, $T_{X}$ and $\omega _{X}$ 
the N\'{e}ron-Severi lattice, the transcendental lattice and a nowhere vanishing 
holomorphic $2$-form on $X$, respectively.

An automorphism of $X$ is \textit{symplectic} 
if it acts trivially on $\mathbb{C} \omega _{X}$.
This paper is devoted to study of \textit{non}-symplectic automorphisms 
of prime-power order for which act trivially on $S_{X}$.  
The study of non-symplectic automorphisms of $K3$ surfaces was pioneered by V.V.~Nikulin.

We suppose that $\varphi $ is a non-symplectic automorphism of order $I$ on $X$ such that 
$\varphi ^{\ast }\omega _{X} =\zeta _{I} \omega _{X}$ where $\zeta _{I}$ is a primitive $I$-th root of unity.
Then $\varphi ^{\ast }$ has no non-zero fixed vectors in $T_{X}\otimes \mathbb{Q}$ and hence  
$\Phi (I)$ divides $\rk T_{X}$, where $\Phi $ is the Euler function.  
In particular $\Phi (I)\leq \rk T_{X}$ and hence $I\leq 66$ 
\cite[Theorem 3.1 and Corollary 3.2]{Ni2}.

The following proposition was announced by Vorontsov \cite{Vo} and then 
it was proved by Kondo \cite{Kondo1}. 

\begin{prop}\label{vo}
Let $\varphi $ be a non-symplectic automorphism on $X$ which acts trivially on $S_{X}$.
Then the order of $\varphi $
is prime-power; 
$p^{k}=2^{\alpha } \ (1\leq \alpha \leq 4)$, $3^{\beta } \ (1\leq \beta \leq 3)$, $5^{\gamma } \ (1\leq \gamma \leq 2)$, 
7, 11, 13, 17 or 19.
Moreover $S_{X}$ is a $p$-elementary lattice, that is,
$S_{X}^{\ast }/S_{X}$ is a $p$-elementary group where $S_{X}^{\ast }=\Hom (S_{X},\mathbb{Z})$.
\end{prop}

Non-symplectic automorphisms of prime order have been studied by several authors e.g.
Nikulin \cite{Ni3}, Oguiso, Zhang \cite{13-19}, \cite{11},
Artebani, Sarti \cite{AS} and Taki \cite{Taki}.
Recently, we have the classification of 
non-symplectic automorphisms of prime order on $K3$ surfaces \cite{AST}.

\begin{thm}\cite[Theorem 1.2]{AST}\label{cla-orderp}
We assume that $S_{X}$ is $p$-elementary.
Let $r$ be the Picard number of $X$  
and let $a$ be the minimal number of generators of $S_{X}^{\ast }/S_{X}$.

Then there exists a non-symplectic automorphism $\varphi $ of order $p$ on $X$ if and only if
$22-r -(p-1)a \in 2(p-1)\mathbb{Z}_{\geq 0}$.

Moreover if $X$ has a non-symplectic automorphism $\varphi $ of order $p$ which acts trivially on $S_{X}$  
then the fixed locus $X^{\varphi }:=\{x\in X| \varphi (x)=x \}$ has the form 
\begin{equation*}
X^{\varphi }=
\begin{cases}
\phi  & \hspace{-2cm} \text{if $S_{X}=U(2)\oplus E_{8}(2)$}, \\
C^{(1)}\amalg C^{(1)} & \hspace{-2cm} \text{if $S_{X}=U\oplus E_{8}(2)$}, \\
\{ P_{1}, \dots , P_{M} \} \amalg C^{(g)} \amalg E_{1} \amalg \dots \amalg E_{N} & \text{otherwise},
\end{cases}
\end{equation*}
and 

\[ g=\frac{22-r-(p-1)a}{2(p-1)}, \] 

\begin{equation*}
M = 
\begin{cases}
0 & \text{if $p=2$}, \\
\dfrac{(p-2)r +22}{p-1} &  \text{if $p=17, 19$}, \\
\dfrac{(p-2)r -2}{p-1} & \text{otherwise},
\end{cases}
\end{equation*}

\begin{equation*}
N = 
\begin{cases}
\dfrac{r-a}{2} & \text{if $p=2$}, \\
0 &  \text{if $p=17, 19$}, \\
\dfrac{2+r-(p-1)a}{2(p-1)} & \text{otherwise}, 
\end{cases}
\end{equation*}
where $P_{j}$ is an isolated point, 
$C^{(g)}$ is a non-singular curve with genus $g$ and 
$E_{k}$ is a non-singular rational curve.
\end{thm}

On the other hand, studies of prime power order have progressed, too. 
Sch\"{u}tt \cite{Schutt} classified $K3$ surfaces with non-symplectic
automorphisms whose order is 2-power and equals $\rk T_{X}$.

Kondo \cite{Kondo1} and Machida and Oguiso \cite{machida-oguiso} 
or Oguiso and Zhang \cite{13-19} have proved that 
the $K3$ surface with non-symplectic automorphisms of order 25 or 27,
respectively, is unique.
Recently, Taki \cite{Taki2} classified non-symplectic automorphisms of 3-power order. 
The following theorem is known.

\begin{thm}
\begin{itemize}
\item[(1)]
$X$ has a non-symplectic automorphism $\varphi$ of order 9 acting trivially on $S_{X}$ 
if and only if $S_{X}= U\oplus A_{2}$, $U\oplus E_{8}$, $U\oplus E_{6}\oplus A_{2}$ or $U\oplus E_{8}\oplus E_{6}$.
Moreover the fixed locus $X^{\varphi }$ has the form
\begin{equation*}
X^{\varphi }=
\begin{cases}
\{ P_{1}, P_{2}, \dots , P_{6} \} &  \text{if $S_{X}=U\oplus A_{2}$,} \\
\{ P_{1}, P_{2}, \dots , P_{10} \}\amalg E_{1} & \text{if $S_{X}=U\oplus E_{8}$ or $U\oplus E_{6}\oplus A_{2}$,}\\
\{ P_{1}, P_{2}, \dots , P_{14} \}\amalg E_{1}\amalg E_{2} & \text{if $S_{X}=U\oplus E_{8}\oplus E_{6}$.}
\end{cases}
\end{equation*}

\item[(2)]
$X$ has a non-symplectic automorphism $\varphi$ of order 27 acting trivially on $S_{X}$ 
if and only if $S_{X}= U\oplus A_{2}$.
Moreover the fixed locus $X^{\varphi }$ has the form
$X^{\varphi}=\{ P_{1}, P_{2}, \dots , P_{6} \}$.
\end{itemize}
Here we denote by $P_{i}$ an isolated point and by $E_{j}$ a non-singular rational curve.
\end{thm}

By Proposition \ref{vo}, if the order of a non-symplectic automorphism is non-prime-power 
then $S_{X}$ is unimodular.
The cases are studied by Kondo \cite{Kondo1}.

\begin{thm}\label{unimodular}
Let $\varphi$ be a non-symplectic automorphism on $X$ 
and $\Phi $ the Euler function.
\begin{itemize}
\item[(1)] If $S_{X}=U$, then $\ord \varphi |$66, 44 or 12.

\item[(2)] If $S_{X}=U \oplus E_{8}$, then $\ord \varphi |$42, 36 or 28.

\item[(3)] If $S_{X}=U \oplus E_{8}^{\oplus 2}$, then $\ord \varphi |$12.

\item[(4)] If $\Phi (\ord \varphi )=\rk T_{X}$, then $\ord \varphi=$
66, 44, 42, 36, 28 or 12. 
Moreover for $m=$ 66, 44, 42, 36, 28 or 12, there exists a unique (up to isomorphisms)
$K3$ surface with $\ord \varphi =m$.
\end{itemize}
\end{thm}

Hence, in order to classify non-symplectic automorphisms on $X$ which act trivially on $S_{X}$, 
we need the complete classification of non-symplectic
automorphisms of 2-power order, i.e., generalization of Sch\"{u}tt's result.
The main purpose of this paper is to prove the following theorem.
See Section \ref{class} for some notations.

\begin{mt}\label{mt}
We assume that $S_{X}$ is 2-elementary.

\begin{itemize}
\item[(1)]
$X$ has a non-symplectic automorphism $\varphi$ of order 4 acting trivially on $S_{X}$ 
if and only if $S_{X}$ has $\delta _{S_{X}}=0$ and 
$S_{X} \neq U\oplus E_{8}(2)$, $U(2)\oplus E_{8}(2)$, $U\oplus D_{4}^{\oplus 3}$ and $U\oplus D_{8}^{\oplus 2}$.

Moreover the fixed locus $X^{\varphi }$ has the form
\begin{equation*}
X^{\varphi }=
\begin{cases}
\{ P_{1}, P_{2}, \dots , P_{4} \} & \text{if $\rk S_{X}=2$,} \\
\{ P_{1}, P_{2}, \dots , P_{6} \} \amalg E_{1} & \text{if $\rk S_{X}=6$,} \\
\{ P_{1}, P_{2}, \dots , P_{8} \} \amalg E_{1} \amalg E_{2} & \text{if $\rk S_{X}=10$,}\\
\{ P_{1}, P_{2}, \dots , P_{10} \} \amalg E_{1} \amalg E_{2} \amalg E_{3} & \text{if $\rk S_{X}=14$,}\\
\{ P_{1}, P_{2}, \dots , P_{12} \} \amalg E_{1} \amalg E_{2} \amalg E_{3} \amalg E_{4} & \text{if $\rk S_{X}=18$.}
\end{cases}
\end{equation*}

\item[(2)]
$X$ has a non-symplectic automorphism $\varphi$ of order 8 acting trivially on $S_{X}$ 
if and only if $S_{X}=U\oplus D_{4}$, $U(2)\oplus D_{4}$ or $U\oplus D_{4}\oplus E_{8}$.
Moreover the fixed locus $X^{\varphi }$ has the form
\begin{equation*}
X^{\varphi }=
\begin{cases}
\{ P_{1}, P_{2}, \dots , P_{6} \}\amalg E_{1} & \text{if $\rk S_{X}=6$,}\\
\{ P_{1}, P_{2}, \dots , P_{12} \}\amalg E_{1}\amalg E_{2} &  \text{if $\rk S_{X}=14$.} 
\end{cases}
\end{equation*}

\item[(3)]
$X$ has a non-symplectic automorphism $\varphi$ of order 16 acting trivially on $S_{X}$ 
if and only if $S_{X}=U\oplus D_{4}$ or $U\oplus D_{4}\oplus E_{8}$.
Moreover the fixed locus $X^{\varphi }$ has the form
\begin{equation*}
X^{\varphi }=
\begin{cases}
\{ P_{1}, P_{2}, \dots , P_{6} \}\amalg E_{1} & \text{if $S_{X}=U\oplus D_{4}$,}\\
\{ P_{1}, P_{2}, \dots , P_{12} \}\amalg E_{1}\amalg E_{2} &  \text{if $S_{X}=U\oplus D_{4}\oplus E_{8}$.} 
\end{cases}
\end{equation*}

\end{itemize}
Here, $P_{i}$ is an isolated point and $E_{j}$ is a non-singular rational curve.
\end{mt}

We summarize the contents of this paper.
In Section \ref{class}, we review the classification of even indefinite 2-elementary lattices. 
And we check the non-existence of lattice isometries of order 4.
As a result, we get the N\'{e}ron-Severi lattice of $K3$ surfaces with 
non-symplectic automorphisms of order 4, 8 or 16 which act trivially on $S_{X}$. 
Section \ref{pre} is a preliminary section.
We recall some basic results about non-symplectic automorphisms on $K3$ surfaces.
Section \ref{4} is the main part of this paper.
Here, we classify non-symplectic automorphisms of order 4.
By using the Lefschetz formula and the classification of non-symplectic involutions,
we study fixed locus of non-symplectic automorphisms of order 4.
In Section \ref{8} and Section \ref{16}, we treat non-symplectic automorphisms of order 8 and 16, respectively.
In Section \ref{example}, we collect examples of $K3$ surfaces with a non-symplectic automorphism of 2-power order.

\begin{ack}
The author would like to express his gratitude to Professor Shigeyuki Kondo for 
giving him many useful comments and informing Example \ref{u2d4d4}.
He also thanks Professors Hisanori Ohashi, Alessandra Sarti and Matthias Sch\"{u}tt 
for pointing out some mistakes and valuable advices.
He is grateful to Professor JongHae Keum for warm encouragement.
His research is supported by Basic Science Research Program through the National Research Foundation(NRF)
of Korea funded by the Ministry of education, Science and Technology (2007-C00002).
And He would like to thank to the referee for pointing
out some mistakes and useful comments.
\end{ack}

\section{The N\'{e}ron-Severi and $p$-elementary lattices}\label{class}

A lattice $L$ is a free abelian group of finite rank $r$ equipped with 
a non-degenerate symmetric bilinear form, which will be denoted by $\langle \ , \ \rangle $.
The bilinear form $\langle \ , \ \rangle $ determines a 
canonical embedding $L\subset L^{\ast }=\Hom (L,\mathbb{Z})$. 
We denote by $A_{L}$ the factor group $L^{\ast }/L$ which is a finite abelian group.
$L(m)$ is the lattice whose bilinear form is the one on $L$ multiplied by $m$. 

We denote by $U$ the hyperbolic lattice defined by $\begin{pmatrix} 0 & 1 \\ 1 & 0 \end{pmatrix}$ 
which is an even unimodular lattice of signature $(1,1)$, and by $A_{m}$, $D_{n} $ or $E_{l}$ an even 
negative definite lattice associated with the Dynkin diagram of type $A_{m}$, $D_{n} $ or $E_{l}$ 
($m\geq 1$,  $n\geq 4$ and $l=6,7,8$). 

Let $p$ be a prime number.
A lattice $L$ is called \textit{$p$-elementary} if $A_{L}\simeq (\mathbb{Z}/p\mathbb{Z})^{\oplus a}$,
where $a$ is the minimal number of generators of $A_{L}$.
For a $p$-elementary lattice we always have the inequality $a\leq r$, since 
$\mid L^{\ast }/L\mid =p^{a}$, $\mid L^{\ast }/pL^{\ast }\mid =p^{r}$ 
and $pL^{\ast }\subset L\subset L^{\ast }$. 

\begin{ex}
For all $p$, lattices $E_{8}$, $E_{8}(p)$, $U$ and $U(p)$ are $p$-elementary.
$A_{1}$, $D_{4}$, $D_{8}$ and $E_{7}$ are 2-elementary.
\end{ex}

\begin{df}
For a 2-elementary lattice $L$, we put
\begin{equation*}
\delta _{L}=
\begin{cases}
0 & \text{if} \ x^{2}\in \mathbb{Z},  \forall x\in L^{\ast }, \\
1 & \text{otherwise}.
\end{cases}
\end{equation*}
\end{df}

Even indefinite 2-elementary lattices were classified by \cite[Theorem 3.6.2]{Ni1}.
\begin{thm}\label{cla-2el}
An even indefinite 2-elementary lattice $L$ is determined by 
the invariants $(\delta _{L}, t_{+},t_{-}, a )$ where the pair $(t_{+},t_{-})$ is the signature of $L$.
\end{thm}

By the Theorem, we can get the N\'{e}ron-Severi lattice of $K3$ surfaces with 
a non-symplectic automorphism of order $2^{k}$ acting trivially on $S_{X}$. 
See Table \cite[Table 1]{Ni3}.

If $k\geq 2$ then $\Phi(2^{k})$ is even.
Since $\Phi(2^{k})$ divides $\rk T_{X}$, $\rk T_{X}$ is even.
Hence if $X$ has a non-symplectic automorphisms of 2-power order  
then $\rk S_{X}$ is even.
Moreover we have the following.

\begin{prop}\label{delta}
Let $L$ be a 2-elementary lattice.
If $\delta _{L}=1$ then $L$ has no non-trivial isometries $f$ of order 4
which act trivially on $A_{L}$ and do not have eigenvalues 1 or $-1$.
\end{prop}

\begin{proof}
Let $f:L\rightarrow L$ be an isometry of order 4 which acts trivially on $A_{L}$
and does not have eigenvalues 1 or $-1$. 
Since the induced isometry $A_{L}\rightarrow A_{L} \ (\bar{x}\mapsto \overline{f^{\ast }(x)})$ is identity,  
for all $x\in L^{\ast }$, there exists an $l \in L$ such that $f^{\ast }(x)=x+l$.

By the assumption, we have $f^{\ast }+f^{\ast 3}=0$.
This implies 
$0= \langle f^{\ast}(x)+f^{\ast 3}(x), x \rangle
= \langle f^{\ast }(x), x \rangle + \langle f^{\ast 3}(x), x \rangle 
=2 \langle f^{\ast }(x), x \rangle 
=2 (\langle x, x \rangle + \langle l, x \rangle )$.
Thus we have $\langle x, x \rangle = -\langle l, x \rangle \in \mathbb{Z}$.
Hence $\delta _{L}=0$.
\end{proof}

The following tables are lists of even 2-elementary lattices with 
an isometry of order 2 and $\delta=0$. 
Hence if $X$ has a non-symplectic automorphisms of order 4, 8 or 16 which act trivially on $S_{X}$ 
then $S_{X}$ is one of the lattices in the following table.
(See also Lemma \ref{sayou} (1).)

\begin{longtable}{|c|c|c|c|}
\hline
$\rk S_{X}$ & $a$ & $S_{X}$ & $T_{X}$ \\
\hline
2 & 0 & $U$ & $U^{\oplus 2}\oplus E_{8}^{\oplus 2}$ \\
\hline
2 & 2 & $U(2)$ & $U\oplus U(2)\oplus E_{8}^{\oplus 2}$ \\
\hline
6 & 2 & $U\oplus D_{4}$ & $U^{\oplus 2}\oplus E_{8}\oplus D_{4}$ \\
\hline 
6 & 4 & $U(2)\oplus D_{4}$ & $U(2)^{\oplus 2}\oplus E_{8}\oplus D_{4}$ \\
\hline 
10 & 0 & $U\oplus E_{8}$ & $U^{\oplus 2}\oplus E_{8}$  \\
\hline
10 & 2 & $U\oplus D_{8}$ & $U^{\oplus 2}\oplus D_{8}$  \\
\hline
10 & 4 & $U\oplus D_{4}^{\oplus 2}$ & $U^{\oplus 2}\oplus D_{4}^{\oplus 2}$  \\
\hline
10 & 6 & $U(2)\oplus D_{4}^{\oplus 2}$ & $U\oplus U(2)\oplus D_{4}^{\oplus 2}$ \\
\hline
10 & 8 & $U\oplus E_{8}(2)$ & $U^{\oplus 2}\oplus E_{8}(2)$ \\
\hline
10 & 10 & $U(2)\oplus E_{8}(2)$ & $U\oplus U(2)\oplus E_{8}(2)$  \\
\hline
14 & 2 & $U\oplus E_{8}\oplus D_{4}$ & $U^{\oplus 2}\oplus D_{4}$ \\
\hline
14 & 4 & $U\oplus D_{8}\oplus D_{4}$ & $U\oplus U(2)\oplus D_{4}$ \\
\hline
14 & 6 & $U\oplus D_{4}^{\oplus 3}$ & $U(2)^{\oplus 2}\oplus D_{4}$  \\
\hline
18 & 0 & $U\oplus E_{8}^{\oplus 2}$ & $U^{\oplus 2}$ \\
\hline
18 & 2 & $U\oplus E_{8}\oplus D_{8}$ & $U\oplus U(2)$ \\
\hline
18 & 4 & $U\oplus D_{8}^{\oplus 2}$ & $U(2)^{\oplus 2}$  \\
\hline
\caption[1]{2-elementary lattices}\label{table-2el}
\end{longtable}

\begin{rem}\label{elpss}
Let $ \{ e, f \}$ be a basis of $U$ (resp. $U(2)$) with $\langle e, e \rangle = \langle f, f \rangle =0$ 
and $\langle e, f \rangle =1$ (resp. $\langle e, f \rangle =2$ ) .
If necessary replacing $e$ by $\varphi (e)$, 
where $\varphi $ is a composition of reflections induced from non-singular 
rational curves on $X$, we may assume that $e$ is represented by the class of an elliptic curve $F$ 
and the linear system $|F|$ defines an elliptic fibration $\pi :X\rightarrow \mathbb{P}^{1}$.
Note that $\pi $ has a section $f-e$ in case $U$. 
In case $U(2)$, there are no $(-2)$-vectors $r$ with $\langle r, e \rangle=1$, and hence $\pi$ has no sections.
\end{rem}

It follows from Remark \ref{elpss} and Table \ref{table-2el} that 
$X$ has an elliptic fibration $\pi :X\rightarrow \mathbb{P}^{1}$.
In the following, we fix such an elliptic fibration.

The following lemma follows from \cite[$\S 3$ Corollary 3]{PSS} and the 
classification of singular fibers of elliptic fibrations \cite{Kodaira}. 
\begin{lem}\label{elliptic}
Assume that $S_{X}=U(m)\oplus K_{1}\oplus \dots \oplus K_{r}$, 
where $m=$ 1 or 2, and $K_{i}$ is a lattice isomorphic to $A_{m}$, $D_{n} $ or $E_{l}$.
Then $\pi$ has a reducible singular fiber with corresponding Dynkin diagram $K_{i}$.
\end{lem}

\section{Preliminaries}\label{pre}

\begin{lem}\label{sayou}
Let $\varphi $ be a non-symplectic automorphism of 2-power order on $X$.
Then we have :
\begin{itemize}
\item[(1)] $\varphi ^{\ast }\mid T_{X}\otimes \mathbb{C}$ can be diagonalized as:
\[ \begin{pmatrix} 
\zeta I_{q} & 0 & \cdots & \cdots & \cdots & 0 \\ 
0 & \zeta^{3} I_{q} &  &  &  & \vdots \\ 
\vdots &  & \ddots &  &  & \vdots \\ 
\vdots &  &  & \zeta^{n} I_{q} &  & \vdots \\ 
\vdots &  &  &  & \ddots  & 0 \\ 
0 & \cdots & \cdots & \cdots & 0 & \zeta^{2k-1} I_{q} \\ 
\end{pmatrix}, \]
where $I_{q}$ is the identity matrix of size $q$, $\zeta$ is a primitive $2^{k}$-th root of unity, 
$n$ is a odd number.
\item[(2)] Let $P$ be an isolated fixed point of $\varphi $ on $X$. Then 
$\varphi ^{\ast }$ can be written as 
\[ \begin{pmatrix}  \zeta ^{i} & 0 \\ 0 & \zeta ^{j}  \end{pmatrix}  \hspace{5mm} (i+j\equiv 1 \mod 2^{k}) \]
under some appropriate local coordinates around $P$.
\item[(3)] Let $C$ be an irreducible curve in $X^{\varphi }$ and $Q$ a point on $C$. 
Then $\varphi ^{\ast }$ can be written as
\[ \begin{pmatrix}  1 & 0 \\ 0 & \zeta   \end{pmatrix} \] 
under some appropriate local coordinates around $Q$. 
In particular, fixed curves are non-singular.
\end{itemize}
\end{lem}

\begin{proof}
(1) This follows form \cite[Theorem 3.1]{Ni2}. 

(2), (3) 
Since $\varphi ^{\ast }$ acts on $H^{0}(X, \Omega _{X}^{2})$ as a multiplication by $\zeta $, 
it acts on the tangent space of a fixed point as 
\[ \begin{pmatrix} 1 & 0 \\ 0 & \zeta  \end{pmatrix} \hspace{1cm} \text{or}  \hspace{1cm} 
\begin{pmatrix} \zeta ^{i} & 0 \\ 0 & \zeta ^{j}  \end{pmatrix} \]
where $i+j\equiv 1 \pmod {2^{k}}$.
\end{proof}
Thus the fixed locus of $\varphi $ consists of a disjoint union of non-singular curves and isolated points. 
Hence we can express the irreducible decomposition of $X^{\varphi }$ as 
\[ X^{\varphi } =\{ P_{1}, \dots , P_{M} \} \amalg C_{1} \amalg \dots \amalg C_{N}, \] 
where $P_{j}$ is an isolated point and $C_{k}$ is a non-singular curve.

In the following, we assume that $k\geq 2$.
Hence we treat non-symplectic automorphisms of order 4, 8 and 16.

\begin{lem}\label{top-euler}
Let $r$ be the Picard number of $X$ and  
$\varphi $ a non-symplectic automorphism of 
2-power order which acts trivially on $S_{X}$.
Then $\chi (X^{\varphi })=r+2$.
\end{lem}

\begin{proof}
We apply the topological Lefschetz formula:
\[ \chi (X^{\varphi }) =\sum _{i=0}^{4}(-1)^{i} \text{tr}(\varphi ^{\ast }|H^{i}(X, \mathbb{R})). \]

Since $\varphi ^{\ast }$ acts trivially on $S_{X}$, $\text{tr}(\varphi ^{\ast }|S_{X})=r $.
By Lemma \ref{sayou} (1), 
$\text{tr}(\varphi ^{\ast }|T_{X})= q(\zeta +\zeta^{3}+\dots + \zeta^{n}+\dots +\zeta^{2k-1})
=-q(1+\zeta^{2}+\dots +\zeta^{2k-2})=0$. 
Hence we can calculate the right -hand side of the Lefschetz formula as follows:
$\sum _{i=0}^{4}(-1)^{i} \text{tr}(\varphi ^{\ast }|H^{i}(X, \mathbb{R})) 
= 1-0+\text{tr}(\varphi ^{\ast }|S_{X})+\text{tr}(\varphi ^{\ast }|T_{X})-0+1
= r+2$.
\end{proof}

\section{Order 4}\label{4}

We shall study the fixed locus of non-symplectic automorphisms of order 4.
In this section, let $\varphi $ be a non-symplectic automorphism of order 4.

\begin{prop}\label{4-point}
Let $r$ be the Picard number of $X$. 
Then the number of isolated fixed points of $\varphi $, 
$M$ is $(r+6)/2$.
\end{prop}

\begin{proof}
First we calculate the holomorphic Lefschetz number $L(\varphi )$ in two ways as in 
\cite[page 542]{AS1} and \cite[page 567]{AS2}. That is 
\[ L(\varphi ) = \sum _{i=0}^{2} \text{tr}(\varphi ^{\ast }|H^{i}(X, \mathcal{O}_{X})), \]
\[ L(\varphi ) = \sum _{j=1}^{M} a(P_{j}) + \sum _{l=1}^{N}b(C_{l}).\]

Here 
\begin{align*}
a(P_{j}) : & =\frac{1}{\det (1-\varphi ^{\ast }|T_{P_{j}})}  \\
 & =\frac{1}{\det \left ( \begin{pmatrix} 1 & 0 \\ 0 & 1 \end{pmatrix} -  \begin{pmatrix} \zeta ^{2} & 0 \\ 0 & \zeta ^{3} \end{pmatrix} \right ) }, 
\end{align*}
\begin{align*}
b(C_{l}) : & =\frac{1-g(C_{l})}{1-\zeta }- \frac{\zeta C_{l}^{2}}{(1-\zeta )^{2}}, 
\end{align*}
where $T_{P_{j}}$ is the tangent space of $X$ at $P_{j}$, $g(C_{l})$ is the genus of $C_{l}$.

Using the Serre duality $H^{2}(X, \mathcal{O}_{X})\simeq H^{0}(X,\mathcal{O}_{X}(K_{X}))^{\vee }$, we calculate 
from the first formula that $L(\varphi )=1+\zeta ^{3}$. 
From the second formula, we obtain
\[
L(\varphi ) = \frac{M}{(1-\zeta ^{2})(1-\zeta ^{3})}
+ \sum_{l=1}^{N} \frac{(1+\zeta )(1-g(C_{l}))}{(1-\zeta )^{2}}. 
\]
Combing these two formulae, we have 
$M=4+\sum _{l=1}^{N}(2-2g(C_{l})$.
By $\chi (X^{\varphi })=M+\sum _{l=1}^{N}(2-2g(C_{l}))$ and 
Lemma \ref{top-euler}, we have $M=(r+6)/2$.
\end{proof}

\begin{prop}\label{ne-order4}
If $S_{X}=U\oplus E_{8}(2)$, $U(2)\oplus E_{8}(2)$, 
$U\oplus D_{4}^{\oplus 3}$ 
or $U\oplus D_{8}^{\oplus 2}$  then 
$X$ has no non-symplectic automorphisms of order 4 which act trivially on $S_{X}$.
\end{prop}
\begin{proof}
We will check the statement for each $S_{X}$ individually.

We assume $S_{X}=U\oplus E_{8}(2)$ or $U(2)\oplus E_{8}(2)$.
If $X$ has a non-symplectic automorphism $\varphi$ of order 4 which acts trivially on $S_{X}$
then $X^{\varphi }$ contains non-singular rational curves 
by Lemma \ref{top-euler} and the proof of Proposition \ref{4-point}.
Although these curves are fixed by $\varphi ^{2}$, it is a contradiction by 
Theorem \ref{cla-orderp}.
This settles Proposition \ref{ne-order4} in cases $S_{X}=U\oplus E_{8}(2)$ and $U(2)\oplus E_{8}(2)$.

We assume $S_{X}=U\oplus D_{4}^{\oplus 3}$ and $X$ has a non-symplectic 
automorphism $\varphi$ of order 4 which acts trivially on $S_{X}$.
Then $X^{\varphi ^{2}}=C^{(1)}\amalg E_{1}\amalg \cdots \amalg E_{4}$
by Theorem \ref{cla-orderp}.

Since $\varphi$ acts trivially on $S_{X}$, 
$\varphi$ preserves reducible singular fibers of an elliptic fibration $\pi$. 
Hence $\varphi $ acts trivially on the base 
of $\pi$ and the section (c.f. Remark \ref{elpss})
is fixed by $\varphi$.
By Lemma \ref{elliptic}, $\pi$
has three singular fibers of type I$_{0}^{\ast }$.
The component with multiplicity 2 is pointwise fixed by $\varphi$.
Hence $X^{\varphi }$ contains at least four non-singular rational curves.
 
On the other hand 
$\chi (C^{(g)} \amalg E_{1} \amalg \dots \amalg E_{N})=16-10=6$ 
by Lemma \ref{top-euler} and Proposition \ref{4-point}. 
Thus $X^{\varphi }$ contains a non-singular curve $C^{(g)}$ with $g\geq 2$.
But this is a contradiction because $X^{\varphi ^{2}}$ does not contain $C^{(g)}$ with $g\geq 2$.
This settles Proposition \ref{ne-order4} in cases $S_{X}=U\oplus D_{4}^{\oplus 3}$.

By \cite[Theorem 1]{Schutt}, 
$X$ with $S_{X}=U\oplus D_{8}^{\oplus 2}$ has no non-symplectic automorphisms 
of order 4 .
\end{proof}

In other cases of Table \ref{table-2el}, there exist $K3$ surfaces with a non-symplectic automorphism of order 4.
See Section \ref{example}.

In the following, we shall describe 
$X^{\varphi}= \{ P_{1},  \dots , P_{M} \} \amalg C^{(g)} \amalg E_{1} \amalg \dots \amalg E_{N}$.

\begin{prop}\label{m-order4}
Assume $S_{X}$ is 2-elementary and $\delta =0$.
If $S_{X} \neq U\oplus E_{8}(2)$, $U(2)\oplus E_{8}(2)$, $U\oplus D_{4}^{\oplus 3}$ or $U\oplus D_{8}^{\oplus 2}$
then  $X^{\varphi }$ has the form
\begin{equation*}
X^{\varphi }=
\begin{cases}
\{ P_{1}, P_{2}, \dots , P_{4} \} & \text{if $\rk S_{X}=2$,} \\
\{ P_{1}, P_{2}, \dots , P_{6} \} \amalg E_{1} & \text{if $\rk S_{X}=6$,} \\
\{ P_{1}, P_{2}, \dots , P_{8} \} \amalg E_{1} \amalg E_{2} & \text{if $\rk S_{X}=10$,}\\
\{ P_{1}, P_{2}, \dots , P_{10} \} \amalg E_{1} \amalg E_{2} \amalg E_{3} & \text{if $\rk S_{X}=14$,}\\
\{ P_{1}, P_{2}, \dots , P_{12} \} \amalg E_{1} \amalg E_{2} \amalg E_{3} \amalg E_{4} & \text{if $\rk S_{X}=18$.}
\end{cases}
\end{equation*}
\end{prop}

\begin{proof}
We will check the form of $X^{\varphi }$ for each $S_{X}$ individually.

Assume $S_{X}=U$.
By Theorem \ref{cla-orderp}, 
$X^{\varphi ^{2}}=C^{(10)}\amalg E_{1}$.
If $X^{\varphi }$ contains a non-singular rational curve $E_{2}$ 
or a non-singular curve $C^{(1)}$ 
then $E_{2}$ or $C^{(1)}$ are also contained $X^{\varphi ^{2}}$. 
This is a contradiction.
Thus $X^{\varphi }$ contains at most one non-singular rational curve 
and no non-singular curves with genus 1. 
Put $X^{\varphi}= \{ P_{1},  \dots , P_{M} \} \amalg C^{(g)} \amalg E_{1} \amalg \dots \amalg E_{N}$.
Then $\chi (C^{(g)} \amalg E_{1} \amalg \dots \amalg E_{N})=4-4=0$ 
by Lemma \ref{top-euler} and Proposition \ref{4-point}. 
If $X^{\varphi }$ contains $E_{1}$ then 
$X^{\varphi }$ contains a non-singular curve $C^{(2)}$.
But this is a contradiction because $X^{\varphi ^{2}}$ does not contain $C^{(2)}$.
Hence $X^{\varphi }=\{ P_{1}, P_{2}, \dots , P_{4} \}$.
This settles Proposition \ref{m-order4} in the case $S_{X}=U$.

Assume $S_{X}=U\oplus E_{8}\oplus D_{4}$.
Then $X^{\varphi ^{2}}=C^{(3)}\amalg E_{1}\amalg \cdots \amalg E_{6}$
by Theorem \ref{cla-orderp}.
We remark that 
$\chi (C^{(g)} \amalg E_{1} \amalg \dots \amalg E_{N})=16-10=6$ 
by Lemma \ref{top-euler} and Proposition \ref{4-point}. 
If $X^{\varphi}$ contains $C^{(3)}$ then 
$X^{\varphi}=\{ P_{1}, P_{2}, \dots , P_{10} \} \amalg C^{(3)}\amalg E_{1}\amalg \cdots \amalg E_{5}$.
Since $E_{6}$ is not fixed by $\varphi$, isolated fixed points $P_{i}$ lie on $E_{6}$.
But this is a contradiction because a non-singular rational curve has exactly two fixed points.
Hence $X^{\varphi }=\{ P_{1}, P_{2}, \dots , P_{10} \} \amalg E_{1} \amalg E_{2} \amalg E_{3}$.
This settles Proposition \ref{m-order4} in the case $S_{X}=U\oplus E_{8}\oplus D_{4}$.

In the other case we can check the claim  
by similar arguments. 
\end{proof}

\section{Order 8}\label{8}
In this section, let $\varphi $ be a non-symplectic automorphism of order 8.
And we shall describe 
$X^{\varphi}= \{ P_{1},  \dots , P_{M} \} \amalg C^{(g)} \amalg E_{1} \amalg \dots \amalg E_{N}$.

\begin{prop}\label{8-point}
Let $r$ be the Picard number of $X$. Then the number of isolated points $M$ is 
$(3r+6)/4$.
\end{prop}

\begin{proof}

Let $P^{i,j}$ be an isolated fixed point given by the local action 
$\begin{pmatrix}  \zeta ^{i} & 0 \\ 0 & \zeta ^{j}  \end{pmatrix}$
and $m_{i,j}$ the number of isolated fixed points of type $P^{i,j}$.

By the holomorphic Lefschetz formulae, we have 
\begin{align*}
\begin{cases}\tag{$\sharp $}\label{8-ten}
0 & = 2m_{3,6} -m_{4,5}-\sum_{l=1}^{N}(2-2g(C_{l})),\\
2 & = m_{2,7}-m_{3,6} +m_{4,5}-\sum_{l=1}^{N}(2-2g(C_{l})).
\end{cases}
\end{align*}

We remark that
$\varphi^{2}(P^{u,v})$ is a fixed point of a non-symplectic automorphism of order 4.
It is easy to see  that 
$\varphi^{2}(P^{2,7})$ and $\varphi^{2}(P^{3,6})$ are isolated fixed points of $\varphi^{2}$.
By Proposition \ref{4-point} and Lemma \ref{8-point-lem}, we have 
\begin{equation}\label{2736}
m_{2,7}+m_{3,6}=\frac{r+6}{2}. 
\end{equation} 

By (\ref{8-ten}), (\ref{2736}) and Lemma \ref{top-euler}, 
we have $M=(3r+6)/4$.
\end{proof}

\begin{lem}\label{8-point-lem}
Let $P$ be an isolated fixed point of $\varphi^{2}$.
Then $\varphi(P) = P$.
\end{lem}
\begin{proof}
Let $m\neq 0$ be the number of such $P$.
Then $m$ satisfies $m_{2,7}+m_{3,6}+m=(r+6)/2$.
By the equation and (\ref{8-ten}), we have 
$m_{2,7}=(r+14)/4-3m/2$, $m_{3,6}=(r-2)/4+m/2$, $m_{4,5}=(r-6)/4+3m/2$ 
and $\sum_{l=1}^{N}(2-2g(C_{l}))=(r+2)/4-m/2$.

Since $m_{2,7}+m_{3,6}$ is even by (\ref{8-ten}), 
$m$ is even, $m_{2,7}$ and $m_{3,6}$ are odd.
Hence we have $m\leq (r+6)/2-1-1=(r+2)/2$.
By the parity of $m_{2,7}$, $m_{3,6}$ and $m_{4,5}$, 
if $r=$ 2, 10 and 18 (resp. 6 and 14) then 
$m=2\times $ odd number (resp. $2\times $ even number).

Assume $r=10$. Then $m=2$ or 6.
If $m=6$ then $m_{2,7}=6-9<0$. This is a contradiction.
If $m=2$ then $m_{4,5}=4$ and $\sum_{l=1}^{N}(2-2g(C_{l}))=2$.
Since $\varphi^{2}(P^{4,5})$ is a point on a irreducible fixed curve by $\varphi^{2}$, 
these two equations imply that $\varphi^{2}$ has 3 fixed 
non-singular rational curves.
This is a contradiction by Proposition \ref{m-order4}.
This settles Lemma \ref{8-point-lem} in the case $r=10$.

In other cases we can check the claim by similar the argument.
\end{proof}

\begin{rem}
$m_{2,7}=(r+14)/4$, 
$m_{3,6}=(r-2)/4$, 
$m_{4,5}=(r-6)/4$.
\end{rem}

\begin{cor}\label{cor-8}
If $X$ has a non-symplectic automorphism of order 8 then 
$\rk S_{X}=$ 6 or 14.
\end{cor}
\begin{proof}
If $\rk S_{X}=$ 2, 10 or 18 then $M$ is odd by  Proposition \ref{8-point}.
But $\chi (X^{\varphi })=M+\sum _{l=1}^{N}(2-2g(C_{l}))$ is even by Lemma \ref{top-euler}.
\end{proof}

If $S_{X}= U\oplus D_{4}$ or $U(2)\oplus D_{4}$ then 
there exist $K3$ surfaces with non-symplectic automorphisms of order 8 
by Example \ref{ud4} and \ref{u2d4}.
And Sch\"{u}tt \cite[Theorem 1]{Schutt} also determines the lattice
$S_{X}$ of rank 14 explicitly. 

\begin{prop}\label{m-order8}
$X$ has a non-symplectic automorphism $\varphi$ of order 8 acting trivially on $S_{X}$ 
if and only if $S_{X}= U\oplus D_{4}$, $U(2)\oplus D_{4}$ or $U\oplus D_{4}\oplus E_{8}$.
Moreover the fixed locus $X^{\varphi }$ has the form
\begin{equation*}
X^{\varphi }=
\begin{cases}
\{ P_{1}, P_{2}, \dots , P_{6} \}\amalg E_{1} & \text{if $\rk S_{X}=6$,}\\
\{ P_{1}, P_{2}, \dots , P_{12} \}\amalg E_{1}\amalg E_{2} &  \text{if $\rk S_{X}=14$.} 
\end{cases}
\end{equation*}
\end{prop}
\begin{proof}
Note $\chi (C^{(g)} \amalg E_{1} \amalg \dots \amalg E_{N})=(2+r)/4$ 
by Lemma \ref{top-euler} and Proposition \ref{8-point}.
We remark that $X^{\varphi ^{2}}$ does not contain non-singular curve with genus $\geq 1$ 
by Proposition \ref{m-order4}.
Thus $N=(2+r)/8$.
\end{proof}

\section{Order 16}\label{16}
In this section, let $\varphi $ be a non-symplectic automorphism of order 16.
And we shall describe $X^{\varphi}= \{ P_{1},  \dots , P_{M} \} \amalg C^{(g)} \amalg E_{1} \amalg \dots \amalg E_{N}$.
We remark that, by Corollary \ref{cor-8}, that
if $X$ has a non-symplectic automorphism of order 16 then $\rk S_{X}=$ 6 or 14.

\begin{prop}\label{16-point}
Let $r$ be the Picard number of $X$. Then the number of isolated points $M$ is 
$(3r+6)/4$.
\end{prop}
\begin{proof}
It is similar to the proof of Proposition \ref{8-point}.
\end{proof}

\begin{rem}
$m_{2,15}=(r+10)/4$, 
$m_{3,14}=(r+2)/8$, 
$m_{4,13}=(r-6)/8$, 
$m_{5,12}=(r-6)/8$, 
$m_{6,11}=(r-6)/8$, 
$m_{7,10}=1$, 
$m_{8,9}=0$.
\end{rem}

Sch\"{u}tt \cite[Theorem 1]{Schutt} proved that 
the $K3$ surface with a non-symplectic automorphism of order 16 
and $\rk S_{X}=6$ is unique 
and that $S_{X}= U\oplus D_{4}$.

By Proposition \ref{m-order8}, 
if $X$ has a non-symplectic automorphism of order 16 and $\rk S_{X}=14$
then $S_{X}= U\oplus D_{4}\oplus E_{8}$.
Indeed there exists a $K3$ surface with non-symplectic automorphisms of order 16 
and $S_{X}= U\oplus D_{4}\oplus E_{8}$. 
See Example \ref{ud4e8}.

\begin{prop}\label{m-order16}
$X$ has a non-symplectic automorphism $\varphi$ of order 16 acting trivially on $S_{X}$ 
if and only if $S_{X}= U\oplus D_{4}$ or $U\oplus D_{4}\oplus E_{8}$.
Moreover the fixed locus $X^{\varphi }$ has the form
\begin{equation*}
X^{\varphi }=
\begin{cases}
\{ P_{1}, P_{2}, \dots , P_{6} \}\amalg E_{1} & \text{if $S_{X}=U\oplus D_{4}$,}\\
\{ P_{1}, P_{2}, \dots , P_{12} \}\amalg E_{1}\amalg E_{2} &  \text{if $S_{X}=U\oplus D_{4}\oplus E_{8}$.} 
\end{cases}
\end{equation*}
\end{prop}
\begin{proof}
It is similar to the proof of Proposition \ref{m-order8}.
\end{proof}

\section{Examples}\label{example}

In this section, we give examples of $K3$ surfaces with a non-symplectic automorphism of 2-power order.
We remark that these $K3$ surfaces have an elliptic fibration from Remark \ref{elpss} and Table \ref{table-2el}.

\begin{ex}\label{u}[\cite[(3.1)]{Kondo1}](\textbf{Case: $S_{X}=U$}) 

$X:y^{2}=x^{3}+x+t^{11}$, 
$\varphi (x, y, t)=(-x,\zeta _{4}y, -t)$.
\end{ex}

\begin{ex}\label{u2}(\textbf{Case: $S_{X}=U(2)$}) 
We do not have an explicit example of $S_{X}=U(2)$ 
though it seems likely that such examples exist.

For example, let $([x_{0}: x_{1}], [y_{0}: y_{1}])$ be the 
bi-homogeneous coordinates on $\mathbb{P}^{1} \times \mathbb{P}^{1}$ 
and $\iota$ an involution of $\mathbb{P}^{1} \times \mathbb{P}^{1}$ given by
$([x_{0}: x_{1}], [y_{0}: y_{1}]) \mapsto ([x_{0}: -x_{1}], [y_{0}: -y_{1}])$.
We remark that $\iota$ has 4 isolated fixed points.
Consider a smooth divisor $C$ in $\mathbb{P}^{1} \times \mathbb{P}^{1}$ of bidegree $(4,4)$ 
such that $f(x_{0},-x_{1},y_{0},-y_{1})=-f(x_{0},x_{1},y_{0},y_{1})$
where $f$ is the defining equation of $C$.
Let $X$ the double cover of $\mathbb{P}^{1} \times \mathbb{P}^{1}$ 
branched along $C$. 
Then $X$ is a $K3$ surface with $U (2)\subset S_{X}$ and for generic $f$ as above we
expect that  $S_{X} = U (2)$.
And the involution $\iota$ induces an automorphism $\varphi$ 
which satisfies $\varphi ^{\ast }\omega_X = \zeta _{4} \omega_{X}$.
\end{ex}

\begin{ex}\label{ud4}[\cite{Schutt}](\textbf{Case: $S_{X}=U\oplus D_{4}$}) 

$X:y^{2}=x^{3}+t^{2}x+t^{11}$, 
$\varphi (x, y, t)=(\zeta _{16}^{2}x,\zeta _{16}^{3} y, \zeta _{16}^{2} t)$.
\end{ex}

\begin{ex}\label{u2d4}[\cite[Proposition 4 (15)]{machida-oguiso}](\textbf{Case: $S_{X}=U(2)\oplus D_{4}$})

Let $X$ be the minimal resolution of the surface 
$\widetilde{X}:=\{z^{2}=x_{0}(x_{0}^{4}x_{2}+x_{1}^{5}-x_{2}^{5})\}$ 
having 5 ordinary double points $[0:1:\zeta_{5}^{i}:0]$ ($i=0, 1, 2, 3, 4$) 
and $\varphi ([x_{0}:x_{1}:x_{2}:z])=[x_{0}:\zeta_{4}x_{1}:\zeta_{4}x_{2}:\zeta_{8}^{5}z]$.
\end{ex}

\begin{ex}\label{ue8}[\cite[(3.2)]{Kondo1}](\textbf{Case: $S_{X}=U\oplus E_{8}$}) 

$X:y^{2}=x^{3}-t^{5}\prod_{i=1}^{6}(t-\zeta_{6}^{i})$, 
$\varphi (x, y, t)=(-x,\zeta _{4}y, -t)$.
\end{ex}

\begin{ex}\label{ud8}(\textbf{Case: $S_{X}=U\oplus D_{8}$}) 

$X:y^{2}=x^{3}+t\prod_{i=1}^{6}(t-\zeta_{6}^{i})x^{2}+t\prod_{i=1}^{6}(t-\zeta_{6}^{i})$, 
$\varphi (x, y, t)=(-x, \zeta _{4}y, -t)$.
\end{ex}

\begin{ex}\label{ud4d4}(\textbf{Case: $S_{X}=U\oplus D_{4}^{\oplus 2}$}) 

$X:y^{2}=x^{3}-t^{3}\prod_{i=1}^{6}(t-\zeta_{6}^{i})$, 
$\varphi (x, y, t)=(-x, \zeta _{4}y, -t)$.
\end{ex}

\begin{ex}\label{u2d4d4}[\cite[\S 2.1]{Kondo2}](\textbf{Case: $S_{X}=U(2)\oplus D_{4}^{\oplus 2}$}) 

Let $\{ [\lambda_{i} : 1]\}$ be a set of distinct 8 points on the projective line.
Let $([x_{0}: x_{1}], [y_{0}: y_{1}])$ be the  bi-homogeneous coordinates on $\mathbb{P}^{1} \times \mathbb{P}^{1}$.  
Consider a smooth divisor $C$ in $\mathbb{P}^{1} \times \mathbb{P}^{1}$ of bidegree $(4,2)$ given by
\[
y_{0}^{2} \cdot \prod^{4}_{i=1} (x_{0} - \lambda_{i} x_{1}) + y_{1}^{2} \cdot \prod^{8}_{i=5} (x_{0} - \lambda_{i} x_{1}) = 0.
\]
Let $L_{0}$ (resp. $L_{1}$) be the divisor defined
by $y_{0} = 0$ (resp. $y_{1} = 0$).
Let $\iota$ be an involution of 
$\mathbb{P}^{1} \times \mathbb{P}^{1}$ given by
\[
([x_{0}: x_{1}], [y_{0}: y_{1}])  \mapsto ([x_{0}: x_{1}], [y_{0}: -y_{1}])
\]
which preserves $C$, $L_{0}$ and $L_{1}$.  

Note that the double cover of $\mathbb{P}^{1} \times \mathbb{P}^{1}$ 
branched along $C + L_{0} +L_{1}$ has 8 rational double points of type
$A_{1}$ and its minimal resolution $X$ is a $K3$ surface. 
The involution $\iota$ lifts to an automorphism $\varphi$ 
which satisfies $\varphi ^{\ast }\omega_X = \zeta _{4} \omega_{X}$.
\end{ex}

\begin{ex}\label{ud4e8}(\textbf{Case: $S_{X}=U\oplus E_{8}\oplus D_{4}$}) 

$X:y^{2}=x^{3}+t^{2}x+t^{7}$, 
$\varphi (x, y, t)=(\zeta _{16}^{10}x,\zeta _{16}^{7} y, \zeta _{16}^{2} t)$.
\end{ex}

\begin{ex}\label{ud4d8}(\textbf{Case: $S_{X}=U\oplus D_{8}\oplus D_{4}$}) 

$X:y^{2}=x^{3}+t\prod_{i=1}^{4}(t-\zeta_{4}^{i})x^{2}+t^{3}\prod_{i=1}^{4}(t-\zeta_{4}^{i})$, 
$\varphi (x, y, t)=( -x,\zeta _{4} y, -t)$.
\end{ex}

\begin{ex}\label{ue8e8}[\cite[(3.4)]{Kondo1}](\textbf{Case: $S_{X}=U\oplus E_{8}^{\oplus 2}$})

$X:y^{2}=x^{3}-t^{5}(t-1)(t+1)$, 
$\varphi (x, y, t)=(-x,\zeta _{4}y, -t)$.
\end{ex}

\begin{ex}\label{u2e8e8}[\cite{Schutt}](\textbf{Case: $S_{X}=U\oplus E_{8}\oplus D_{8}$}) 

$X:y^{2}=x^{3}+tx^{2}+t^{7}$, 
$\varphi (x, y, t)=( -x,\zeta _{4} y, -t)$.
\end{ex}

\begin{rem}
Assume that an elliptic $K3$ surface $\pi :X\rightarrow \mathbb{P}^{1}$  
is given by a Weierstrass equation.
Then it is easy to see types of singular fibers of $\pi$ 
by the discriminant and the $j$-invariant.
And we have the rank of the Mordell-Weil group of $\pi$ by \cite[\S 5]{Shioda}.
By the Shioda-Tate formula 
$\rk S_{X}=2+\rk M.W.+
\sum _{F:\text{fiber}}(\sharp \{ \text{components of } F \}-1)$, 
we can determine $S_{X}$.
See also Lemma \ref{elliptic}.

For example, in Example \ref{ud4e8}, $\pi :X\rightarrow \mathbb{P}^{1}$ 
has reducible singular fibers of type II$^{\ast }$ and of type I$_{0}^{\ast }$.
It follows that the rank of the Mordell-Weil group is 0 and 
$\rk S_{X}=2+0+(9-1)+(5-1)=14$.
\end{rem}

\end{document}